\documentclass[11pt]{article}
\usepackage{geometry}
\usepackage{epsfig}
\usepackage{amsfonts}
\usepackage{amsmath}
\usepackage{amssymb}
\usepackage{amsthm}
\usepackage{graphicx,color}
\usepackage{fancyhdr}
\usepackage[usenames,dvipsnames]{xcolor}
\usepackage{enumerate}
\usepackage{caption}
\usepackage{subcaption}
\allowdisplaybreaks
 
\newtheorem{theorem}{Theorem}
\newtheorem{corollary}{Corollary}
\newtheorem{proposition}{Proposition}
\newtheorem{definition}{Definition}
\newtheorem{lemma}{Lemma}
\newtheorem{remark}{Remark}

\newcommand{\RR}{\mathbb{R}}

\newcommand{\EEE}{\mathcal{E}}

\newcommand{\be}{\begin{equation}}
\newcommand{\ee}{\end{equation}}
\newcommand{\bea}{\begin{eqnarray}}
\newcommand{\eea}{\end{eqnarray}}
\newcommand{\bean}{\begin{eqnarray*}}
\newcommand{\eean}{\end{eqnarray*}}
\newcommand{\ben}{\begin{equation*}}
\newcommand{\een}{\end{equation*}}
\newcommand{\ea}{\end{aligned}}

\newcommand{\PP}{\textbf{P}}

\newcommand{\origin}{\boldsymbol 0}
\newcommand{\ba}{\boldsymbol a}
\newcommand{\NN}{\mathbb{N}}
\newcommand{\ZZ}{\mathbb{Z}}
\newcommand{\MM}{\mathbb{M}}
\newcommand{\CC}{\mathbb{C}}
\newcommand{\cone}{\mathbb{E}}
\newcommand{\bv}{\textbf{v}}

\newcommand{\bw}{\textbf{w}}

\newcommand{\dd}{\mathrm{d}}

\newcommand{\sid}[1]{{\color{black} #1}}
\newcommand{\wtd}[1]{{\color{black} #1}}

\numberwithin{equation}{section}
\numberwithin{theorem}{section}
\numberwithin{corollary}{section}
\numberwithin{proposition}{section}
\numberwithin{lemma}{section}
\numberwithin{remark}{section}
\numberwithin{figure}{section}
\numberwithin{table}{section}
\usepackage{pdfsync}

\begin{document}
\title{\sc Multivariate Regular Variation of Discrete Mass Functions with Applications to Preferential Attachment Networks}
\author{\sc Tiandong Wang and Sidney I. Resnick \thanks{\sid{Tiandong Wang
    and S. Resnick} were supported by Army MURI grant W911NF-12-1-0385 to Cornell University.}}

\maketitle

\begin{abstract}
Regular variation of a \sid{multivariate} measure \sid{with a Lebesgue density }
implies the regular variation of its density \sid{provided} the density satisfies
some regularity conditions. 
\sid{Unlike the univariate case,}
the converse \sid{also requires} regularity conditions. 
We extend these arguments to discrete mass functions and their
associated measures \sid{using the concept} that the 
the mass function can be embedded
in a continuous density function. We  give two different
conditions, monotonicity and \wtd{convergence on the unit sphere}, both of which can make the
discrete function embeddable. \sid{Our results are} then applied to
the preferential attachment network model, and we conclude that the
joint mass function of in- and out-degree is embeddable and thus
regularly varying. 
\end{abstract}
\bigskip{\bf Keywords:} Multivariate regular variation, preferential
attachment, random graphs, power laws, in-degree, out-degree.
\section{Introduction}
The influence of heavy tailed modeling methods has spread to many
fields. Application areas for the modeling and statistical methods
include finance \cite{Smith2003}, insurance
\cite{EKM1997},
 social networks and random graphs
 \cite{Durrett2010,BBCR03,RS15,SRTDWW16}, mobility modeling for
 wireless phone users \cite{KKS2015}, parallel processing queueing
 models of cloud computing \cite{JTWST2013}, models to optimize power
 usage when a mobile user changes between wifi and mobile networks
 \cite{KLS2014}. 

The theory of regular variation is \sid{an} essential
\sid{mathematical tool}
in the analysis of heavy
tailed \sid{phenomena}.  
A measurable function $f$ is regularly varying with index $\alpha > 0$ (written $f\in RV_\alpha$) if $f : \RR_+ \mapsto\RR_+ $ satisfies 
\[\lim_{t\rightarrow\infty} \frac{f(tx)}{f(t)} = x^\alpha,\quad\mbox{ for } x > 0.\] 
%We can and do assume such functions are continuous and strictly increasing.
In the univariate case, Karamata's theorem gives the asymptotic
properties of the \sid{indefinite} integral; differentiation of a \sid{regularly varying
integral to recover the density function is covered by the monotone
density theorem}; see \cite[Chapter 2.3]{resnickbook:2007}
and \cite[page 38]{BGT1987}.
\sid{Roughly}, if the
derivative of $f$ is regularly varying with index $\alpha$, then $f$
 is regularly varying with index $\alpha+1$. Conversely, \sid{when $f$ is}
 monotone, the
derivative of $f\in RV_{\alpha}$ is regularly varying with index
$\alpha-1$. \sid{These results can be interpreted as relating regular
variation properties of an absolutely continuous  measure and the
density of the measure.}

\sid{In practice, collected data are often multidimensional which raises the
issue of generalizing the one-dimensional theory relating an integral
and derivative to higher dimensions. Unfortunately additional
regularity conditions are indispensible for the generalizations.}  This is
discussed in a series of
publications \cite{dHR1979, dHOR1984, dHR1987,dHO1984,O1989}.
\sid{An additional issue is that in the era of Internet and social network
studies, many data sets are discrete. 
For example, when studying the growth of social networks,
the in- and out-degrees of nodes exhibit power laws.} Under such
circumstances, ignoring the discrete essence of the data is not
appropriate \sid{and we need to understand the relationship between regular
variation properties of a discretely supported 
measure and its mass function.}

\sid{We proceed using the idea of {\it embedding\/} a function of a
discrete variable in a function of a continuous variable.
In one dimension (\cite{BS1973}, \cite[Theorem 1.9.5]{BGT1987}),
a  regularly varying
sequence can always be embedded in a regularly varying function of a
continous variable.}
 The
embeddability property in \sid{the} multivariate case is not as obvious because
the definition of regular variation exerts no control over the
function's variation when we move from ray to ray \sid{and a}dditional
conditions are required for embeddability. Once this is resolved the
results in the continuous cases can be used to examine the relationship
between the regular variation of a mass function and its associated
measure. 

We  apply our results to the preferential attachment network model
studied in \cite{BBCR03, krapivsky:redner:2001,SRTDWW16,RS15,wang:resnick:2015} where a new node
attaches to an existing node or new edges 
are created according to probabilistic postulates that take into
account the current in- and out-degrees of the existing nodes.
\sid{The joint asymptotic distribution of in- and out-degree
has multivariate regularly varying tails \cite{RS15,SRTDWW16}}.
We check the embeddability
conditions for the joint mass function of in- and out-degree and
conclude that the mass 
function is also regularly varying.

This paper is organized as follows. We start with a brief overview of
the multivariate regular variation of both measures and functions in 
Section~\ref{prelim} and then state the univariate embeddability
results in Section~\ref{onedim}. We describe 
the preferential attachment model in
Section~\ref{model}. Sections~\ref{monot} and \ref{unif} provide two
different conditions to establish embeddability in the bivariate case
and also characterize the relationship between the regular variation
of a mass function and its measure. We then apply our results to the
preferential attachment model in Sections~\ref{appl1} and \ref{appl2},
verifying that the joint mass function of in- and out-degrees is also
regularly varying.

\section{Preliminaries}
\subsection{Multivariate Regular Variation}\label{prelim}
We briefly review the basic concepts of multivariate regular variation
for measures with emphasis on the two dimensional case. We use 
$\MM$-convergence to define regular variation instead of
the traditional way of using vague convergence. See \cite{HL2006,das:mitra:resnick:2013,  LRR2014} for the details on 
$\MM$-convergence and \sid{reasons for its use.  }

Consider $\RR_+^2$ metrized by a  convenient metric $d(x,y)$. A subset $\CC\subseteq\RR_+^2$ is a cone if it is closed under positive scalar multiplication: if $\bv\in\CC$ then $c\bv\in\CC$ for $c > 0$. A proper framework for discussing regular variation is measure convergence on a closed cone $\CC\subseteq\RR_+^2$ with a closed cone $\CC_0\subseteq\CC$ deleted. $\CC_0$ is called the \emph{forbidden zone}. In this paper we are interested in the case where $\CC=\RR_+^2$ and $\CC_0=\{\origin\}$. Then $\cone=\RR_+^2\setminus\{\origin\}$ is the space for defining $\MM$-convergence appropriate for regular variation of distributions of positive random vectors. The forbidden zone is the origin $\{\origin\}$.

Let $\MM(\CC\setminus\CC_0)$ be the set of Borel measures on
$\CC\setminus\CC_0$ which are finite on sets bounded away from the
forbidden zone $\CC_0$. We now give the definition of
$\MM$-convergence which becomes the basis for our definition of
multivariate regular variation of measures. 

\begin{definition}
For $\mu_n,\mu\in\MM(\CC\setminus\CC_0)$ we say $\mu_n\rightarrow\mu$ in $\MM(\CC\setminus\CC_0)$ if $\int f\dd\mu_n\rightarrow\int f\dd \mu$ for all bounded, continuous, non-negative $f$ on $\CC\setminus\CC_0$ whose support is bounded away from $\CC_0$.
\end{definition}

A random vector $(X,Y)\ge 0$ is non-standard regularly varying on $\CC\setminus \CC_0$ if there exists $b_1(t)\in RV_{1/\alpha_1}$ and $b_2(t)\in RV_{1/\alpha_2}$ ($\alpha_1,\alpha_2>0$), called the scaling functions, and a measure $\nu(\cdot) \in\MM(\CC\setminus \CC_0)$, called the limit or tail measure, such that as $t\rightarrow\infty$,
\[
t\PP\left[\left(\frac{X}{b_1(t)},\frac{Y}{b_2(t)}\right)\in\cdot\right]\rightarrow \nu(\cdot),\quad \mbox{in }\MM(\CC\setminus \CC_0).
\]
When $b_1(t)=b_2(t)$, $(X,Y)$ is said to have a distribution with
\emph{standard} regularly varying tails with index
$\alpha:=\alpha_1=\alpha_2$ and the limiting measure $\nu$ satisfies
the scaling property: $\nu(c\cdot)=c^{-\alpha}\nu(\cdot)$ for
$c>0$. \sid{Without loss of generality,} we assume  all \sid{scaling} functions are continuous and strictly increasing.

Following the definition in \cite{dHOR1984},
 we say a measurable function $f:\RR_+^2\mapsto \RR_+$ is
 \emph{multivariate regularly varying with scaling functions $b_1$ and
   $b_2$ and limit function $\lambda$}, if there exists
 $h:(0,\infty)\mapsto (0,\infty)$ with
 $h\in RV_{\alpha}$ for some $\alpha\in\RR$ such that 
\be\label{rvSid}
\lim_{t\rightarrow\infty}\frac{f(b_1(t)x,b_2(t)y)}{h(t)}=\lambda(x,y)\sid{>0,\quad
\forall x,y>0.}\ee
If both $b_1$ and $b_2$ are \sid{the} identity function, we get ordinary
regular variation as in \cite{dHR1979}.

\subsection{\sid{Regularly Varying Functions of Discrete Variables} and  Embeddability}\label{onedim}
Internet modeling and social network studies require many relevant
variables to be discrete. A natural example is in- and out-degree of nodes in a
random graph. So it is useful to examine regular variation for
functions of discrete variables.
\sid{For}  the one dimensional case, see \cite{BS1973,BGT1987}.
\begin{definition} 
A sequence $(c_n)_{n\in \NN}$ of positive numbers  is regularly varying with index $\alpha\in\RR$ if 
\be\label{rvseq}
\lim_{n\to\infty}
c_{[nx]}/c_n = x^\alpha>0,\quad x>0.
\ee 
\sid{A doubly indexed function
  $k:\ZZ^2\setminus\{\origin\}\mapsto\RR_+$ is regularly varying with
  scaling functions $b_1$ and $b_2$ and limit function $\lambda(x,y)$
if  for some  $h\in RV_\alpha$ for some $\alpha\in\RR$, $b_i\in
RV_{\beta_i},$ $\beta_i>0$, we have}
\be\label{rvfct}
\lim_{n\rightarrow\infty}\frac{k([b_1(n)x],[b_2(n)y])}{h(n)}=\lambda(x,y)>0,
\quad \forall x,y>0.\ee
\end{definition}
\sid{In one dimension, a regularly varying sequence can always be embedded
in a regularly varying function of a continuous argument} 
\cite[Theorem 1.9.5]{BGT1987}:
\begin{theorem}
 \label{BGT}
If $(c_n)$ is regularly varying in the sense of \eqref{rvseq}, then the function
\begin{equation}\label{emb}
f(x):=c_{[x]},\sid{\quad x>0,}
\end{equation}
varies regularly with index $\alpha$, where $[x]$ denotes the largest integer smaller than or equal to $x$.
\end{theorem}
So  every regularly varying sequence is
\emph{embeddable} meaning that there exists a regularly varying
function $f:\mathbb{R}_+\mapsto \mathbb{R}_+$ such that $c_n=f(n)$,
and  \eqref{emb} implies $c_{n-1}/c_n\rightarrow 1$ as $n\rightarrow\infty$. 
\sid{To extend}  embeddability to the multivariabe function $k(i,j)$,  we need to
make additional assumptions on the 
behavior of the function $k(i,j)$. Similar to \eqref{rvseq}, when $k$
satisfies \eqref{rvfct},
 we say 
$k(i,j)$ is \emph{embeddable} if there exists a bivariate
\sid{regularly varying} function $g(x,y)$ \sid{satisfying \eqref{rvSid}}
 such that $g(x,y):=k([x],[y])$.

\subsection{Preferential Attachment Network Models}\label{model}
The directed \sid{edge} preferential attachment model
\cite{krapivsky:redner:2001,BBCR03}
 is a model
for a growing directed random graph. The model evolves according to
\sid{certain rules.}
Choose strictly positive parameters $\alpha,\beta, \gamma,\lambda,\mu$
such that $\alpha+\beta+\gamma=1$, and additionly assume 
that $\alpha,\beta,\gamma<1$ to avoid trivial cases. 
\sid{The initial condition for the model is} a  finite directed graph,
denoted by $G(n_0)$, with at least one node and $n_0$ edges.  
For $n=n_0+1, n_0+2,\ldots$, $G(n)$ is a graph with $n$ edges and a
random number $N(n)$ of nodes. 
If a node $v$ is from $G(n)$, use $D_{in}(v)$ and $D_{out}(v)$ to denote its
in and out degree respectively (dependence on $n$ is suppressed). Then $G(n+1)$ is obtained from $G(n)$
as follows. 
\begin{enumerate}
\item[(i)] With probability $\alpha$ a new node $w$ is born and we add an edge leading from $w$ to an existing node $v\in G(n)$ (written as $w\mapsto v$). The existing node $v$ is chosen with probability according to its in-degree:
\[\PP(v\in G(n)\mbox{ is chosen})=\frac{D_{in}(v)+\lambda}{n+\lambda N(n)}.\]
\item[(ii)] With probability $\beta$ we add a directed edge $v\mapsto
  w$ between two existing nodes $v,w\in G(n)$. Nodes $v$ and $w$ are chosen
  independently from all the nodes of $G(n)$ with probabilities 
\[\PP(v\mbox{ is chosen})=\frac{D_{out}(v)+\mu}{n+\mu N(n)},\mbox{ and }
\PP(w\mbox{ is chosen})=\frac{D_{in}(w)+\lambda}{n+\lambda N(n)}.
\]
\item[(iii)] With probability $\gamma$ a new node $w$ is born and we add an edge leading from an existing node $v\in G(n)$ to $w$. The existing node $v$ is chosen with probability according to its out-degree:
\[\PP(v\in G(n)\mbox{ is
  chosen})=\frac{D_{out}(v)+\mu}{n+\mu N(n)}.
  \]
\end{enumerate}

For $i,j=0,1,2,\ldots$ and $n\ge n_0$, let $N_{ij}(n)$ be the random
number of nodes in $G(n)$ with in-degree $i$ and out-degree
$j$. According to \cite[Theorem~3.2]{BBCR03}, there exist non-random
constants $p(i,j)$ such that 
\be\label{pij}
\lim_{n\rightarrow\infty}\frac{N_{ij}(n)}{N(n)}=p(i,j)\quad\mbox{a.s. for } i,j=0,1,2,\ldots.
\ee
\sid{Define two random
  variables $(I,O)$ such that 
$$P[I=i, O=j]=p(i,j),\quad i,j=0,1,2,\dots$$
and the distribution generated  by $(I,O)$ is a non-standard regularly varying
measure \cite{RS15,SRTDWW16}.}
\sid{The pair $(I,O)$ has representation}
\be\label{jointIO}
(I,O)\stackrel{d}{=}B(1+X_1,Y_1)+(1-B)(X_2,1+Y_2),
\ee
where $B$ is a Bernoulli switching variable independent of $X_j, Y_j$, $j=1,2$ with
\[\PP(B=1)=1-\PP(B=0)=\frac{\gamma}{\alpha+\gamma}.\]
Let $T_\delta(p)$ be a negative binomial integer valued random variable with parameters $\delta>0$ and $p\in(0,1)$.
Now suppose $\{T_{\delta_1}(p), p\in(0,1)\}$ and
$\{\tilde{T}_{\delta_2}(p), p\in(0,1)\}$ are two independent families
of negative binomial random variables and  define
\begin{align*}
c_1=\frac{\alpha+\beta}{1+\lambda(\alpha+\gamma)},\quad c_2=\frac{\beta+\gamma}{1+\mu(\alpha+\gamma)}
\quad\mbox{and } a=c_2/c_1.
\end{align*}
By \wtd{\cite[Theorem~2]{SRTDWW16}}, $X_j, Y_j$, $j=1,2$ in
\eqref{jointIO} can be written as 
\be\label{rep}
(X_1,Y_1)=(T_{\lambda+1}(Z^{-1}),\tilde{T}_\mu(Z^{-a})), \mbox{ and
}(X_2,Y_2)=(T_{\lambda}(Z^{-1}),\tilde{T}_{\mu+1}(Z^{-a})),
\ee
where $Z$ is a Pareto random variable on $[1,\infty)$ with index
$c_1^{-1}$, independent of the negative binomial random variables. 

We
will show that  $p(i,j)$  in \eqref{pij} is regularly
varying. 

%\section{Results}
\section{Embeddability and  Monotonicity.}\label{monot}
The embeddability problem is no longer \sid{as straightforward in
  higher dimensions as Theorem \ref{BGT} would lead us to believe;}
here we only deal
with the case where $d=2$.  Multivariate regular variation provides 
 no control over the function's variation when we move from ray to
ray and in order to obtain embeddability in the bivariate
case we need regularity conditions.
The following theorem provides one approach.

\begin{theorem}[Standard case]\label{emb1}
Suppose $u:\ZZ_+^2\setminus\{\origin\}\mapsto\RR_+$ is an eventually
decreasing (in both arguments) mass function 
that is regularly varying: There exists
\wtd{$h(\cdot)\in\text{RV}_{\rho}$, $\rho<0$} such that, 
\begin{equation}\label{dmrv}
\lim_{n\rightarrow\infty}\frac{u([nx],[ny])}{h(n)}=\lambda(x,y)>0,
\quad \forall x,y>0.
\end{equation} 
Then the function
 $g(x,y):=u([x],[y])$
is eventually decreasing and regularly varying,
\begin{equation}\label{cmrv}
\lim_{t\rightarrow\infty}\frac{g(tx,ty)}{h(t)}=\lambda(x,y),  \quad  \forall x,y>0,
\end{equation}
so that $u$ is embeddable in the regularly varying function $g$.
\end{theorem}

\begin{remark}\label{rmk1}
{\rm \sid{In particular, if a probability mass function (pmf) $p(i,j)$ plays
 the role of $u$ and satisfies
 monotonicity and  \eqref{dmrv},
then we can embed $p(i,j)$ in a monotone regularly varying probability density
function (pdf) $f(x,y)$ by simply assigning constant  
probability density $p([x],[y])$ over the rectangle $[[x],[x+1]]\times [[y],[y]+1]$.}}
\end{remark}

\begin{proof}
For $s>0$, we  show that 
the limiting
function $\lambda$ in \eqref{dmrv}  satisfies 
$\lambda(sx,sy)=\wtd{s^\rho}\lambda(x,y)$ as is the case for any
function satisfying \eqref{dmrv} and embeddability.
For  fixed $s>0$, \eqref{dmrv} gives
\[\lim_{n\rightarrow\infty}\frac{u([nsx],[nsy])}{h(n)}=\lambda(sx,sy).\]
We rewrite the left hand side as
\be\label{decomp}
\frac{u([nsx],[nsy])}{h(n)}=\frac{u([nsx],[nsy])}{u([[ns]x],[[ns]y])}\times\frac{u([[ns]x],[[ns]y])}{h([ns])}\times\frac{h([ns])}{h(n)}.
\ee
The second term on the right hand side of \eqref{decomp} converges to $\lambda(x,y)$ as $n\rightarrow\infty$, by \eqref{dmrv}, and as $h\in RV_\wtd{\rho}$, the third term converges to $s^\wtd{\rho}$. It now suffices to show 
\be\label{embalt}
\lim_{n\rightarrow\infty}\frac{u([nsx],[nsy])}{u([[ns]x],[[ns]y])}=1.
\ee
Since  $u$ is eventually decreasing,  for $x>0,\,y>0$ and $n$ large
enough $u([nx],[ny])$ is decreasing in both $x$ and $y$.  
Hence,
\be\label{pee}
\frac{u([([ns]+1)x],[([ns]+1)y])}{u([[ns]x],[[ns]y])}\le\frac{u([nsx],[nsy])}{u([[ns]x],[[ns]y])}\le
1,\ee 
and 
\begin{align}
\frac{u([([ns]+1)x],[([ns]+1)y])}{u([[ns]x],[[ns]y])}
&=
  \frac{u([([ns]+1)x],[([ns]+1)y])}{h([ns]+1)}\frac{h([ns]+1)}{h([ns])}\frac{h([ns])}{u([[ns]x],[[ns]y])}
  \label{extrapee}\\ 
& \stackrel{n\rightarrow\infty}{\rightarrow} \lambda(x,y)\cdot 1\cdot \frac{1}{\lambda(x,y)}=1,\nonumber
\end{align}
which coupled with \eqref{pee} shows \eqref{embalt}. Set $g(x,y)=u([x],[y])$ and 
on the one hand, as $t\to\infty$,
\begin{align*}
\frac{g(tx,ty)}{h(t)} \sim & \frac{u([tx],[ty])}{h([t])} \leq
  \frac{u([[t]x],[[t]y])}{h([t])} \\
\to& \lambda(x,y),\\
\intertext{and on the other, since $([t]+1)x\geq tx,$}
  \frac{g(tx,ty)}{h(t)}\geq &                              \frac{g( ([t]+1)x,([t]+1)y)}{h(t)}\\
=&                              \frac{u( ([[t]+1)x],[([t]+1)y)]}{h(t)},\\
\end{align*}
and the result follows from \eqref{extrapee}.
\end{proof}

We  generalize  Theorem~\ref{emb1} to the non-standard case.
\begin{corollary}[Non-standard case]\label{emb1cor}
Suppose $u:\ZZ_+^2\setminus\{\origin\}\mapsto\RR_+$ is \sid{a mass
  function that is eventually decreasing in both arguments.}  \sid{For 
$\alpha_1,\alpha_2>0$,} let
 $b_1\in RV_{1/\alpha_1}$ and $b_2\in RV_{1/\alpha_2}$ 
be two strictly increasing scaling functions such that 
\begin{equation}\label{dmrvns}
\lim_{n\rightarrow\infty}\frac{u([b_1(n)x],[b_2(n)y])}{h(n)}=\lambda(x,y)>0,
\quad \forall x,y>0,
\end{equation} 
where \wtd{$h(\cdot)\in\text{RV}_\rho$, $\rho<0$}. 
The limit function
$\lambda$ satisfies the scaling property:  
$\lambda(s^{1/\alpha_1}x,s^{1/\alpha_2}y)=s^\wtd{\rho}\lambda(x,y)$
for all $s,x,y>0$ and 
$g(x,y)=u([x],[y])$ is
 an eventually decreasing regularly varying function
such that  and for all $x,y>0$ 
\begin{equation}\label{cmrvns}
\lim_{t\rightarrow\infty}\frac{g(b_1(t)x,b_2(t)y)}{h(t)}=\lambda(x,y)>0.
\end{equation}
Thus $u$ is embeddable in \sid{the} regularly varying function $g$.
\end{corollary}

\begin{proof}
By \cite[Theorem~1.2.2]{O1989}, \eqref{dmrvns} implies that for all $x,y>0$,
\[
\lim_{n\rightarrow\infty}\frac{u([b_1(nx)],[b_2(ny)])}{h(n)}=\lambda(x^{1/\alpha_1},y^{1/\alpha_2}).
\]
Since $b_1$ and $b_2$ are strictly increasing,  $u([b_1(x)],[b_2(y)])$ is eventually decreasing in $x$ and $y$.
A similar proof as in Theorem~\ref{emb1} shows that the 
function $g$ \sid{defined by}  $g(b_1(x),b_2(y)):=u([b_1(x)],[b_2(y)])$ \sid{satisfies}
 for all $x,y>0$,
\[
\lim_{t\rightarrow\infty}\frac{g(b_1(tx),b_2(ty))}{h(t)}=\lambda(x^{1/\alpha_1},y^{1/\alpha_2}).
\]
Then the scaling property of $\lambda$ follows immediately. Applying
\cite[Theorem~1.2.2]{O1989}
 again, we  conclude that $u$ can
be embedded in a function $g$ satisfying \eqref{cmrvns}.
\end{proof}

\sid{Assuming a probability measure has a mass function satisfying the
monotonicity condition, 
we make explicit the }
relation between regular variation of the probability measure
and regular variation of its pmf. 
The first part of Theorem \ref{meastodens}
 is the multivariate analogue of the
{\it monotone density theorem\/}.
\begin{theorem}\label{meastodens}
Suppose two non-negative integer valued random variables $(X,Y)$
\sid{have mass function} $ 
p(i,j)$ and $p$ is eventually decreasing in both arguments. 
For  $\alpha_1,\alpha_2>0$, assume further
 that $b_1(\cdot)\in RV_{1/\alpha_1}$, $b_2(\cdot)\in
RV_{1/\alpha_2}$ are strictly increasing, continuous scaling
functions.

\begin{enumerate}
\item {\sc (Regular variation of the measure implies regular variation
  of the mass function:)}
 If there exists a limit
measure $\nu\in\MM(\mathbb{R}_+^2\setminus \{\origin\})$ 
with density $\lambda(x,y)$, such that 
as $t\rightarrow\infty,$
\be\label{convmeas2}
 t\PP\left[\left(\frac{X}{b_1(t)},\frac{Y}{b_2(t)}\right)\in \cdot\right]\rightarrow \nu(\cdot),\qquad\mbox{in }\MM(\mathbb{R}_+^2\setminus \{\origin\}),
\ee
then the mass function $p(i,j)$ is regularly varying:
\be\label{convdens2}
\lim_{t\rightarrow\infty}\frac{p([b_1(t)x],[b_2(t)y])}{(tb_1(t)b_2(t))^{-1}}=\lambda(x,y),
\quad x,y>0.
\ee
\item {\sc (Regular variation of the mass function implies regular
    variation of the measure:)}
Conversely, if $p(i,j)$ satisfies \eqref{convdens2}, then the
distribution of $(X,Y)$ is regularly varying and \eqref{convmeas2}
holds with $\nu(dx,dy)=\lambda(x,y)dxdy.$
\end{enumerate}
\end{theorem}
\begin{proof}
{(1) Extend the pmf $p(i,j)$ to an eventually decreasing
 pdf} $f(x,y)$ using the method in Remark~\ref{rmk1}, such that
$p(i,j)=f(i,j)$ for all $i,j\in\NN$. 
Let $(X^*,Y^*)$ have pdf $f(x,y)$, and \eqref{convmeas2} implies
\[ 
t\PP\left[\left(\frac{X^*}{b_1(t)},\frac{Y^*}{b_2(t)}\right)\in \cdot\right]\rightarrow \nu(\cdot),\qquad\mbox{as }t\rightarrow\infty.
\] 
By \cite[Theorem 2.3]{dHO1984}, we have
\[\lim_{t\rightarrow\infty}\frac{f(b_1(t)x,b_2(t)y)}{(tb_1(t)b_2(t))^{-1}}=\lambda(x,y).\]
Then \eqref{convdens2} can be recovered by noting that
$p([b_1(t)x],[b_2(t)y])=f(b_1(t)x,b_2(t)y)$ and $p(i,j)$ is embeddable
in $f(x,y)$ in the sense of \eqref{cmrv}.

(2) Since $p(i,j)$ satisfies \eqref{convdens2} and is eventually
monotone,  $p(i,j)$ can be embedded into a pdf, $f(x,y)$ which is also
eventually monotone.
Following the arguments in \wtd{\cite{dHR1979, dHO1984}, we have for
  all Borel sets $A\subseteq\RR_+^2\setminus\{\origin\}$ bounded away from
  $\origin$ such that $\int_A\lambda(x,y)\dd x\dd y<\infty$,  
\[
\lim_{t\rightarrow\infty}\int_A\frac{f(b_1(t)x,b_2(t)y)}{(tb_1(t)b_2(t))^{-1}}\dd x\dd y
        =\int_A\lambda(x,y)\dd x\dd y,
\]
and \eqref{convmeas2} then follows by the embeddability of
$p(i,j)$. }
\end{proof}

\subsection{Application to  the Preferential Attachment Model}\label{appl1}
\sid{As an example of the previous results, we treat 
 a special case of the preferential
attachment model}, where $c_1=c_2$, i.e. the joint regular variation of
the distribution of $(I,O)$ is standard. Due to the decomposition in
\eqref{jointIO}, it suffices to verify the monotonicity of the joint
mass function of $(X_j,Y_j)$, $j=1,2$. We will show that the joint
mass function of $(X_1,Y_1)$ is decreasing and the same argument also
works for $(X_2,Y_2)$.  Our approach requires the fact that the gamma
function $\Gamma(x)$ is log-convex and $(\log \Gamma)''>0$. See 
\cite[page 260, Section 6.4]{AS1972}. This means $f(x)=\log \Gamma(x)$
is convex, $f'(x)=\Gamma ' (x)/\Gamma(x)$ and $f''(x)=(\Gamma''
(x)\Gamma(x) -(\Gamma'(x))^2)/\Gamma(x))^2 >0.$

Recall Section \ref{model} and \eqref{rep} and \eqref{jointIO}. When
$c_1=c_2$, we have 
$(X_1,Y_1)=(T_{\lambda+1}(Z^{-1}),\tilde{T}_\mu(Z^{-1}))$. 
Set $q(i,j)=\PP[X_1=i,Y_1=j]$ and 
for $i,j\in \NN$, 
\begin{align}\label{qij}
q(i,j)&:= \PP[T_{\lambda+1}(Z^{-1} )=i, \tilde{T}_\mu(Z^{-1})=j]\nonumber\\
&= \frac{\Gamma(i+\lambda+1)}{\Gamma(\lambda+1)\Gamma(i+1)}\frac{\Gamma(j+\mu)}{\Gamma(\mu)\Gamma(j+1)}\cdot\frac{1}{c_1}
\int_1^\infty z^{-(2+1/c_1+\lambda+\mu)}(1-z^{-1})^{i+j}\dd z\nonumber\\
&=\Biggl(
  \frac{\Gamma(1+1/c_1+\lambda+\mu)}
{c_1\Gamma(\lambda+1)\Gamma(\mu)}\Biggr)\times 
  \frac{\Gamma(i+\lambda+1)/\Gamma(i+1)\cdot
  \Gamma(j+\mu)/\Gamma(j+1)}{\Gamma(i+j+1/c_1+\lambda+\mu+2)/\Gamma(i+j+1)}\nonumber\\ 
&=: C  \frac{\Gamma(i+\lambda+1)/\Gamma(i+1)\cdot \Gamma(j+\mu)/\Gamma(j+1)}{\Gamma(i+j+1/c_1+\lambda+\mu+2)/\Gamma(i+j+1)}.
\end{align}
Taking the log of $q(i,j)$, pretend $i$ is a continuous variable 
 and take the first  partial derivative with respect to $i$. This gives
\begin{align*}
\frac{\partial}{\partial i}\log q(i,j)&=\left[\frac{\Gamma'(i+\lambda+1)}{\Gamma(i+\lambda+1)}-\frac{\Gamma'(i+1)}{\Gamma(i+1)}\right]
-\left[\frac{\Gamma'(i+j+1/c_1+\lambda+\mu+2)}{\Gamma(i+j+1/c_1+\lambda+\mu+2)}-\frac{\Gamma'(i+j+1)}{\Gamma(i+j+1)}\right]\\
&= [f'(i+\lambda+1)-f'(i+1)]-[f'(i+j+1/c_1+\lambda+\mu+2)-f'(i+j+1)],
\end{align*}
where we used the notation $f=\log \Gamma$.
Since $f$ is convex,  $f'$ is increasing on $\RR_+$. For 
fixed $x>0$ and $a>0$,  $g(a):=f'(x+a)-f'(x)$, $a>0$, satisfies 
 $g'(a)=f''(x+a)>0$, 
and thus  $g(a)$ is increasing in $a$. Hence for $j\ge 0$,
\[f'(i+j+1/c_1+\lambda+\mu+2)-f'(i+j+1)\ge f'(i+1/c_1+\lambda+\mu+2)-f'(i+1),\]
which gives
\begin{align*}
\frac{\partial}{\partial i}\log q(i,j) & \le [f'(i+\lambda+1)-f'(i+1)]-[f'(i+1/c_1+\lambda+\mu+2)-f'(i+1)]\\
&= f'(i+\lambda+1)-f'(i+1/c_1+\lambda+\mu+2)<0.
\end{align*}
Therefore, $q(i,j)$, the joint mass function of $(X_1, Y_1)$,
 is decreasing in $i$. Analogously we can
 show that
 $q(i,j)$ is also decreasing in $j$. The monotonicity of the joint
 mass function $p(i,j)$ of $(I,O)$ defined in \eqref{pij} then
 follows from the decomposition in \eqref{jointIO}. Hence we are left
 to show \eqref{dmrv} for $p(i,j)$. 

Using \eqref{qij} with the constant $C$ restored, we have
\begin{align*}
q([nx],[ny])=\Biggl(
\frac{\Gamma(1+1/c_1+\lambda+\mu)}{c_1\Gamma(\lambda+1)\Gamma(\mu)}
\Biggr)
\times \frac{\Gamma([nx]+\lambda+1)/\Gamma([nx]+1)\cdot \Gamma([ny]+\mu)/\Gamma([ny]+1)}{\Gamma([nx]+[ny]+1/c_1+\lambda+\mu+2)/\Gamma([nx]+[ny]+1)}.
\end{align*}
Hence, for all $x,y>0$
\be\label{qlim}
\lim_{n\rightarrow\infty}\frac{q([nx],[ny])}{n^{-(2+1/c_1)}}=\frac{\Gamma(1+1/c_1+\lambda+\mu)}{c_1\Gamma(\lambda+1)\Gamma(\mu)}
\frac{x^\lambda y^{\mu-1}}{(x+y)^{1+\lambda+\mu+1/c_1}},
\ee
which shows that $q(i.j)$ is regularly varying. Similar calculations
can be done for $(X_2, Y_2)$ and  then using \eqref{jointIO}, we have for
all $x,y>0$, 
\begin{align}\label{limpij}
\lim_{n\rightarrow\infty}\frac{p([nx],[ny])}{n^{-(2+1/c_1)}}=\frac{\gamma}{\alpha+\gamma}&\frac{\Gamma(1+1/c_1+\lambda+\mu)}{c_1\Gamma(\lambda+1)\Gamma(\mu)}\frac{x^\lambda y^{\mu-1}}{(x+y)^{1+\lambda+\mu+1/c_1}}\nonumber\\
&+\frac{\alpha}{\alpha+\gamma}\frac{\Gamma(1+1/c_1+\lambda+\mu)}{c_1\Gamma(\lambda)\Gamma(\mu+1)}
\frac{x^{\lambda-1} y^\mu}{(x+y)^{1+\lambda+\mu+1/c_1}}.
\end{align}
Therefore, in the special case that $c_1=c_2$, 
 $p(i,j)$ is embeddable and standard regularly varying.

We also know from \cite[Theorem~2]{SRTDWW16} 
 that the joint distribution (as opposed to the mass function) of $(I,O)$ is regularly varying,
 and Theorem~\ref{meastodens} \sid{implies} that the
 joint mass function $p(i,j)$ must also be regularly varying in the
 sense of \eqref{convdens2}.  
{}From \eqref{limpij}, the limit function for
$p(i,j)$ is exactly the density for the limit measure specified in
\cite{SRTDWW16}. {We summarize:
\begin{proposition}
In the preferential attachment model with $c_1=c_2$, the asymptotic
joint mass function of  
in- and out-degrees $p(i,j)$ is eventually monotone, embeddable and
standard regularly varying and satisfies 
\eqref{limpij}.
\end{proposition}

When $c_1 \neq c_2$, we have not succeeded in demonstrating 
 the monotonicity condition for the mass function. Instead, in the
 next section, 
we will give a different sufficient condition which can be used to verify the regular
variation of the joint mass function.
}

\section{Embeddability and Convergence on the Unit Sphere.}\label{unif}
\sid{In  Section \ref{monot},} embeddability of a mass function is
guaranteed by assuming monotonicity; however, sometimes monotonicity is either
 not applicable or difficult to verify.  An alternate approach
\sid{is to fix a norm $\|\cdot\|$} and suppose regular variation on
$$\aleph_0:=\{\textbf{v}\in \RR_+^2: ||\textbf{v}||=1\},$$
 the unit
sphere \sid{in this norm} relative to the origin,
with respect to a continuous variable.
This is sufficient for embeddability but further uniformity and
boundedness conditions are necessary to relate pmf's, pdf's and their measures.

\begin{theorem}[Standard Case]\label{emb2}
  Suppose $h(\cdot)\in
  RV_\wtd{\rho}$, ${\rho}<0$ and $u:\ZZ_+^2 \mapsto\RR_+$ satisfies:
%\begin{enumerate}
%\item[(i)] For all $x,y>0$   
%\be\label{mrv}\lim_{n\rightarrow\infty}\frac{u([nx],[ny])}{h(n)}=\lambda(x,y)>0;\ee
%\item[(ii)]
there exists a  limit function $\lambda_0 >0$ defined on
$\aleph_0$ such that 
\be\label{pointwise}
\lim_{t\rightarrow\infty}\frac{u([tx],[ty])}{h(t)}
=\lambda_0(x,y),\quad \forall (x,y)\in \aleph_0.
\ee
%\end{enumerate}
Then
\begin{enumerate}
\item The doubly indexed function $u(i,j)$ is regularly varying:
 For all $x,y>0$   
\be\label{mrv}\lim_{n\rightarrow\infty}\frac{u([nx],[ny])}{h(n)}=\lambda(x,y):=
\lambda_0 \Bigl( \frac{(x,y)}{\|(x,y)\|} \Bigr) \|(x,y)\|^\rho >0;\ee

\item  The doubly indexed function 
$u(i,j )$ is embeddable in a \sid{regularly varying} function
  $f:\RR_+^2\mapsto\RR$ with limit function $\lambda(\cdot)$ such that $f(x,y)=u([x],[y])$;
\item If \eqref{pointwise} is uniform on $\aleph_0$, then $f$ also  satisfies
\be\label{unif_f}
\lim_{t\rightarrow\infty}\sup_{(x,y)\in\aleph_0}\left|\frac{f(tx,ty)}{h(t)}-\lambda_0(x,y)\right|=0.
\ee
\end{enumerate}
\end{theorem}

\sid{Note if a pdf satisfies \eqref{unif_f} with $\lambda_0$ positive and
bounded on $\aleph_0$, then also the measure
is regularly varying. See \cite{dHR1987,resnickbook:2008}. }

\begin{proof}
Write $\bv:=(x,y)$, $u([t\bv]):=u([tx],[ty])$, and for $\bv$ such that
$\|\bv\|\neq 0$, we have $\ba:={\bv}/{\|\bv\|}\in\aleph_0$.  We
show that convergence on $\aleph_0$ implies convergence everywhere. 
We have for $\bv =\|\bv\|\cdot \ba \neq \origin$, 
\begin{align*}
\frac{u([t\bv])}{h(t)} =& \frac{u([t\|v\|\cdot \ba])}{h(t)}
=\frac{u([t\|v\|\cdot \ba])}{h(t\|v\|)}
\cdot
\frac{h(t\|v\|)}{h(t)} \to \lambda_0(\ba) \|\bv\|^\rho.
\end{align*}
This gives \eqref{mrv} and regular variation of $f$ follows
directly. If \eqref{pointwise} is uniform on $\aleph_0$ then by
definition of $f(\bv)=u([\bv])$, \eqref{unif_f} is true.
\end{proof}

\sid{There are various possibile ways to extend this result to the
  non-standard case, depending on the purpose in mind. The following
  is crafted with Section \ref{appl2} in mind.}

\begin{corollary}[Non-standard case; power law scaling]\label{emb2cor}
  Suppose $h(\cdot)\in
  RV_\rho$, $\rho<0$,  $u:\ZZ_+^2 \mapsto\RR_+$  and scaling functions
  are power laws; i.e., $b_i(t) =t^{1/\alpha_i},\,i=1,2$. If
there exists  a  limit function $\lambda_0 >0$ defined on
$$\mathcal{E}_0:=\{(x,y): (x^{\alpha_1},y^{\alpha_2})\in\aleph_0\}$$
 such that 
\be\label{pointwisens}
\lim_{t\rightarrow\infty}
\frac{u([t^{1/\alpha_1} x],[t^{1/\alpha_2} y])}{h(t)}
=\lambda_0(x,y), \quad  \forall (x,y)\in \mathcal{E}_0,
\ee
then
\begin{enumerate}
\item The doubly indexed function $u(i,j)$ is regularly varying:
 For all $x,y>0$, define
 $\bw=\bw(x,y):=(x^{\alpha_1},y^{\alpha_2})$ and
\be\label{mrv_ns}\lim_{n\rightarrow\infty}\frac{u([n^{1/\alpha_1}
  x],[n^{1/\alpha_2} y])}{h(n)}=\lambda(x,y):= 
\lambda_0 \Bigl( \frac{x}{\|\bw\|^{1/\alpha_1}},\frac{y}{\|\bw\|^{1/\alpha_2}} \Bigr) \|\bw\|^\rho >0;\ee

\item  The doubly indexed function $u(i,j )$ is embeddable in a non-standard regularly varying function
  $f:\RR_+^2\mapsto\RR$ with limit function $\lambda(\cdot)$ such that $f(x,y)=u([x],[y])$;
\item If convergence in \eqref{pointwisens} is uniform on
  $\mathcal{E}_0$, then also,
\be\label{unif_f_ns}
\lim_{t\rightarrow\infty}\sup_{(x,y)\in\mathcal{E}_0}\left|\frac{f(t^{1/\alpha_1} x,t^{1/\alpha_2} y)}{h(t)}-\lambda_0(x,y)\right|=0.
\ee
\end{enumerate}
\end{corollary}

\begin{remark}\label{bigmouth}
 {\rm
 If a \sid{regularly varying} pdf $f(x,y)$ on $\mathbb{R}_+^2$ satisfies \eqref{unif_f_ns} with $\lambda_0$ positive and
bounded on $\mathcal{E}_0$, then also the measure 
\sid{induced by the pdf} 
is regularly varying. If a pmf $u(i,j)$ satisfies \eqref{pointwisens}
with $\lambda$ positive and bounded on $\mathcal{E}_0$, then the
corresponding discretely supported measure is regularly varying. See
\cite{dHR1987,resnickbook:2008} and the comments after the proof of
Corollary \ref{emb2cor}.
}
\end{remark}

\begin{proof}
Note that for all $(x,y)\neq\origin$, the function $\bw$ creates a map onto
$\mathcal{E}_0$ and
$$(x,y) \mapsto \Bigl( \frac{x}{\|\bw\|^{1/\alpha_1}},\frac{y}{\|\bw\|^{1/\alpha_2}}
\Bigr)\in\mathcal{E}_0.$$
We show that convergence on $\mathcal{E}_0$ implies convergence
for all $x,y>0$: For $x,y >0$,
\begin{align*}
\lim_{t\rightarrow\infty}\frac{u([t^{1/\alpha_1} x],[t^{1/\alpha_2} y])}{h(t)}
&= \lim_{t\rightarrow\infty}\frac{u([t^{1/\alpha_1} \|\bw\|^{1/\alpha_1}\cdot (x/\|\bw\|^{1/\alpha_1})],
        [t^{1/\alpha_2} \|\bw\|^{1/\alpha_2}\cdot
  (y/\|\bw\|^{1/\alpha_2})])}{h(t)}\\
&=\lim_{t\rightarrow\infty}\frac{u(
[(t\|\bw\|)^{1/\alpha_1} x/\|\bw\|^{1/\alpha_1}],
        [ (t\|\bw\|)^{1/\alpha_2}   y/\|\bw\|^{1/\alpha_2}]}{h(t\|\bw\|)}\cdot\frac{h(t\|\bw\|)}{h(t)}\\
&=\lambda_0 \Bigl( \frac{x}{\|\bw\|^{1/\alpha_1}},\frac{y}{\|\bw\|^{1/\alpha_2}} \Bigr) \|\bw\|^\rho,
\end{align*}
which verifies \eqref{mrv_ns}. The embeddability and \eqref{unif_f_ns} follow by a similar
 argument as in the proof of Theorem~\ref{emb2}.
\end{proof}

For Remark \ref{bigmouth}: If we assume $f$ is regularly varying,
$\lambda_0$ is bounded on $\mathcal{E}_0$ and \eqref{unif_f_ns},
it is straightforward 
 to generalize results in \cite[Theorem~2.1]{dHR1987}:   For $(x,y)\in
 A$, where $A$ is a Borel set bounded away from $\origin$, 
 find an integrable bound for 
\[\frac{f((t\|\bw\|)^{1/\alpha_1} \|\bw\|^{-1/\alpha_1}x,(t\|\bw\|)^{1/\alpha_2}\|\bw\|^{-1/\alpha_2}y)}{h(t\|\bw\|)}\cdot\frac{h(t\|\bw\|)}{h(t)},\]
using the boundedness of $\lambda_0$ on $\mathcal{E}_0$ to bound the first term and Karamata's 
representation to bound the second (same as in the proof of \cite[Theorem~2.1]{dHR1987}).
Then the convergence of the associated measure follows from dominated convergence.

\subsection{Application to the Preferential Attachment Model.}\label{appl2}
We \sid{now apply} Corollary~\ref{emb2cor} to show that in
the  case where $c_1\neq c_2$, the 
joint pmf of $(I,O)$ is also regularly varying. 
The following lemma \sid{is a variant of Stirling's formula.}
\begin{lemma}\label{gunif}
For \sid{a} compact set $K\subset \sid{(0,\infty)}$ and \sid{$0<k\neq 1$}, we have 
\[\lim_{t\rightarrow\infty}\sup_{x\in K}
\left|\frac{\Gamma(tx+k)}{t^k\Gamma(tx)}-x^k\right|=0.\] 
\end{lemma}

\begin{proof}
Define $f_t(x)=\frac{\Gamma(tx+k)}{t^k\Gamma(tx)}$. 
By Stirling's formula \cite[p. 254]{AS1972}, 
we have  pointwise convergence:
\be\label{ptwise}\lim_{t\rightarrow\infty}f_t(x)=x^k,\quad \sid{x>0.}\ee
According to \wtd{\cite[Theorem~4]{QCN2006}} or \cite[Theorem
1.3]{QCN2008}, \sid{for $x>0$,}  $\log f_t(x)$ is increasing in $t$ if $k\in(0,1)$ and decreasing in $t$ if $k>1$. 
Either case will allow us to apply Dini's theorem to the pointwise convergence in \eqref{ptwise} to conclude the uniform convergence of $f_t$ on K.
\end{proof}

\sid{Now we show regular variation of $p(i,j)$ for the preferential
 attachment model in the nonstandard case where $c_1\neq c_2$.}
Here we still only detail the calculations for $(X_1, Y_1)$; results
for $(X_2, Y_2)$ can be obtained in a similar way. 

Set $a=c_2/c_1$, $b_i(t)=t^{c_i}$.  We have
\begin{align*}
q([t^{c_1}x],[t^{c_2} y])=&\frac{\Gamma([t^{c_1}
                            x]+\lambda+1)}{\Gamma(\lambda+1)\Gamma([t^{c_1}x]+1)}\frac{\Gamma([t^{c_2}
                            y]+\mu)}{\Gamma(\mu)\Gamma([t^{c_2}y]+1)} \sid{\frac{1}{c_1}}\times\\
&\times\int_1^\infty s^{-(1+1/c_1)}(1-s^{-1})^{[t^{c_1}x]}s^{-(\lambda+1)}(1-s^{-a})^{[t^{c_2}y]}s^{-a\mu}\dd s.
\end{align*}
It suffices to assume according to Corollary \ref{emb2cor} that
$(x,y)\in \mathcal{E}_0$.
Making the change of variable \sid{$z=s/t^{c_1}$}, we rewrite the integral as
\be\label{integral}
t^{-(1+c_1(\lambda+1)+c_2\mu)}\int_{1/t^{c_1}}^\infty z^{-(2+1/c_1+\lambda+a\mu)}\left(1-\frac{z^{-1}}{t^{c_1}}\right)^{[t^{c_1}x]}
\left(1-\frac{z^{-a}}{t^{c_2}}\right)^{[t^{c_2}y]}\dd z.
\ee
\wtd{Note that for any $(x,y)$ on the compact set $\mathcal{E}_0$ defined in Corollary~\ref{emb2cor},
%\be\label{nsexp}
%\left(1-\frac{z^{-1}}{t^{c_1}}\right)^{[t^{c_1}x]}\left(1-\frac{z^{-a}}{t^{c_2}}\right)^{[t^{c_2}y]}\rightarrow e^{-\left(\frac{x}{z}+\frac{y}{z^a}\right)}
%\ee
%uniformly as $t\rightarrow\infty$ \sid{{\Large in what variable over
%    what domain? $z\in [0,\infty)$?? Do we need this to justify taking
 % the limit under the integral or would dominated convergence
 % suffice??}}.  
dominated convergence gives the convergence of the integral in \eqref{integral} to 
\[\int_0^\infty z^{-(2+1/c_1+\lambda+a\mu)}e^{-\left(\frac{x}{z}+\frac{y}{z^a}\right)}\dd z,\]
as $t\rightarrow\infty$.}
Therefore, according to Corollary \ref{emb2cor}
we have for all $x,y>0$,
\be\label{limitqijns}
\lim_{n\rightarrow\infty}\frac{q([n^{c_1}x],[n^{c_2}y])}{n^{-(1+c_1+c_2)}}
= \frac{x^\lambda y^{\mu-1}}{c_1\Gamma(\lambda+1)\Gamma(\mu)}\int_0^\infty z^{-(2+1/c_1+\lambda+a\mu)}e^{-\left(\frac{x}{z}+\frac{y}{z^a}\right)}\dd z,
\ee
which shows that $q(i,j)$ is regularly varying. Applying
\eqref{jointIO} again, we
 conclude,
\begin{align}\label{limitpijns}
\lim_{n\rightarrow\infty}\frac{p([n^{c_1}x],[n^{c_2}y])}{n^{-(1+c_1+c_2)}}
= \frac{\gamma}{\alpha+\gamma}&\frac{x^\lambda y^{\mu-1}}{c_1\Gamma(\lambda+1)\Gamma(\mu)}\int_0^\infty z^{-(2+1/c_1+\lambda+a\mu)}e^{-\left(\frac{x}{z}+\frac{y}{z^a}\right)}\dd z\nonumber\\
&+\frac{\alpha}{\alpha+\gamma}\frac{x^{\lambda-1} y^\mu}{c_1\Gamma(\lambda)\Gamma(\mu+1)}\int_0^\infty z^{-(1+a+1/c_1+\lambda+a\mu)}e^{-\left(\frac{x}{z}+\frac{y}{z^a}\right)}\dd z.
\end{align}

Therefore,  $p(i,j)$ is regularly varying with scaling functions $b_i(t)=t^{c_i}$ for $i=1,2$ and limit function as in \eqref{limitpijns}, which is the density of the limit measure given in \wtd{\cite{SRTDWW16}. We summarize:
%\sid{State result formally or highlight \eqref{limitpijns}.}
\begin{proposition}
In a preferential attachment model where $c_1\neq c_2$, the joint mass function of 
in- and out-degrees $p(i,j)$ is embeddable, nonstandard regularly varying and satisfies
\eqref{limitpijns}.
\end{proposition}
}

\wtd{
Using Lemma~\ref{gunif}, we can also conclude that the convergence in
\eqref{limitqijns} is uniform on $\mathcal{E}_0$ and so is \eqref{limitpijns}.
The limit function specified in \eqref{limitpijns} is positive and bounded on $\mathcal{E}_0$.
 Hence, the results in Corollary~\ref{emb2cor}(3) and
 Remark~\ref{bigmouth} are also applicable here
}\sid{ and allow the conclusion that regular variation of the pmf
  $p(i,j)$ implies regular variation of the associated measure proven
  in \cite{RS15,SRTDWW16}}

\bibliographystyle{plain}
\nocite{*}
\bibliography{bib-DiscreteMRV.v5.bib}

\end{document}